\documentclass[reqno,a4paper,12pt]{amsart}
\usepackage{bbm}
\usepackage{mathtools}

\usepackage{amssymb}
\usepackage[usenames,dvipsnames]{xcolor}
\usepackage{hyperref}
\definecolor{tianyi}{HTML}{1A237E}
\definecolor{secondary}{HTML}{1B5E20}
\definecolor{doi_link}{HTML}{2E7D32}
\definecolor{DarkR}{HTML}{B71C1C}
\hypersetup{
    pdftitle = {The Berry-Ess\'{e}en Upper Bounds of Vasicek Model Estimators},
    pdfauthor = {Yumin Cheng},
    pdfsubject={Stochastic Analysis},
    pdfkeywords={Vasicek model, Malliavin calculus, Central limit theorem, Berry-Ess\'{e}en upper bounds},
    pdfcreator = {Texlive 2021},
    pdfproducer = {Texlive 2021},
    bookmarksopen=true,
    colorlinks=true,
    linkcolor=tianyi,
    anchorcolor=secondary,
    citecolor=tianyi,
    urlcolor=secondary,
    bookmarksnumbered =true,
    pdfdisplaydoctitle = true,
}
\usepackage{cleveref}
\usepackage{mathrsfs}
\usepackage[numbers,sort]{natbib}
\newtheorem{theorem}{Theorem}
\newtheorem{lemma}[theorem]{Lemma}
\newtheorem{proposition}[theorem]{Proposition}

\newtheorem{definition}{Definition}
\newtheorem{hypothesis}{Hypothesis}

\newtheorem*{remarkn}{Remark}

\usepackage{geometry}
\def\bE{Berry-Ess\'{e}en\ }

\newcommand{\norm}[1]{\left\Vert #1\right\Vert}
\newcommand{\normh}[1]{\left\Vert #1\right\Vert_{\mathfrak{H}}}
\newcommand{\abs}[1]{\left\vert #1\right\vert}

\def\Rnum{\mathbb{R}}

\def\mE{\mathbb{E}}

\def\m{\mbox}
\def\Pro{\mathbb{P}}

\def\me{\mathrm{e}}
\def\kls{\hat{k}_{LS}}
\def\muls{\hat{\mu}_{LS}}
\makeatletter
\newcommand\dlmu[2][3.5cm]{\hskip1pt\underline{\hb@xt@ #1{\hss#2\hss}}\hskip3pt}
\makeatother
\newcommand*{\dif}{\mathop{}\!\mathrm{d}}
\newcommand\distribution{\stackrel{law}{\longrightarrow}}
\newcommand{\RomanNum}[1]{\uppercase\expandafter{\romannumeral #1\relax}}
\newcommand{\etal}{{\text{et al.}}}
\newcommand{\iproct}[2]{\left\langle #1,#2\right\rangle}

\begin{document}
\title{The Berry-Ess\'{e}en Upper Bounds of Vasicek Model Estimators}
\author{Yong Chen}
    \address{School of Mathematics and Statistics, Jiangxi Normal University, Nanchang, 330022, Jiangxi, China}\email{zhishi@pku.org.cn}
\author{Yumin Cheng*}
    \address{School of Mathematics and Statistics, Jiangxi Normal University, Nanchang, 330022, Jiangxi, China}\email{chengym@jxnu.edu.cn}

\begin{abstract}
    The \bE upper bounds of moment estimators and least squares estimators of the mean and drift coefficients in Vasicek models driven by general Gaussian processes are studied. When studying the parameter estimation problem of Ornstein-Uhlenbeck (OU) process driven by fractional Brownian motion, the commonly used methods are mainly given by Kim and Park, they show the upper bound of Kolmogorov distance between the distribution of the ratio of two double Wiener-It\^{o} stochastic integrals and the Normal distribution. The main innovation in this paper is extending the above ratio process, that is to say, the numerator and denominator respectively contain triple Wiener-It\^{o} stochastic integrals at most. As far as we know, the upper bounds between the distribution of above estimators and the Normal distribution are novel.
\end{abstract}

\maketitle
\bibliographystyle{unsrtnat}%

\section{Introduction}\label{bib_vasicek}

Vasicek model is a type of 1-dimensional stochastic processes, it is used in various fields, such as economy, finance, environment. It was originally used to describe short-term interest rate fluctuations influenced by single market factors. 
Proposed by O. Vasicek \cite{vasicek_equilibrium_1977}, it is the first stochastic process model to describe the ``mean reversion" characteristic of short-term interest rates. In the financial field, it can also be used as a random investment model in Wu \etal \cite{wu_optimal_2020} and Han \etal \cite{han2021calibrating}.

\begin{definition}\label{general_vasicek}
    Consider the Vasicek model driven by general Gaussian process, it satisfies the following Stochastic Differential Equation (SDE):
    \begin{equation}
        \dif V_t=k(\mu -V_t)\dif t+\sigma \dif G_t,\;\;\; t\in[0,T],
        \label{Vmdif}
    \end{equation}
    where $k, T>0,\ V_0=0$ and $G=\{G_t\}_{t\geq 0}$ is a general one-dimensional centered Gaussian process that satisfies \Cref{hyp_1}. 
\end{definition}

This paper mainly focuses on the convergence rate of estimators of coefficient $k,\mu$. Without loss of generality, we assume $\sigma=1$, then Vasicek model can be represent by the following form:
\begin{equation*}
        V_t=\mu(1-\me^{-kt})+\int^t_0 e^{-k(t-s)}\dif G_s.
        \label{Vmexp}
\end{equation*}

When the coefficients in the drift function is unknown, an important problem is to estimate the drift coefficients based on the observation.
Based on the Brownian motion, Fergusson and Platen \cite{fergusson_application_2015} present the maximum likelihood estimators of coefficients in Vasicek model. When the Vasicek model driven by the fractional Brownian motion, Xiao and Yu \cite{xiao2019asymptotic} consider the least squares estimators and their asymptotic behaviors. When $k>0$, Hu and Nualart \cite{hu_parameter_2010} study the moment estimation problem.

Since the Gaussian process $G_t$ mainly determines the trajectory properties of Vasicek model. Therefore, following the assumptions in Chen and Zhou \cite{chen_parameter_2021}, we make the following Hypothesis about $G_t$.
\begin{hypothesis}[\cite{chen_parameter_2021} Hypothesis 1.1]\label{hyp_1}
    Let $\beta \in (\frac{1}{2},1)$ and $t\neq s\in \lbrack 0,\infty )$, Covariance function $R(t,s)=\mathbb{E}[G_{t}G_{s}]$ of Gaussian process $G_t$ satisfies the following condition:
    \begin{align}
        \frac{\partial ^{2}}{\partial t\partial s}R(t,s)&=C_{\beta}\abs{t-s}^{2\beta-2}+\Psi (t,s),
        \label{cond hyp1}
    \end{align}
    where
    \begin{align*}
        \abs{\Psi (t,s)}&\leq C_{\beta}^{\prime}\abs{ts}^{\beta -1},
    \end{align*}
    $\beta,\ C_{\beta}>0,C_{\beta}^{\prime}\geq 0$ are constants independent with $T$. Besides, $R(0,t)=0$ for any $t\geq 0$.
\end{hypothesis}
\begin{remarkn}
    The covariance functions of Gaussian processes such as fractional Brownian motion, subfractional Brownian motion and double fractional Brownian motion satisfy the above Hypothesis \cite[Examples 1.5-1.8]{chen_parameter_2021}. 
\end{remarkn}

Assuming that there is only one trajectory $(V_t, {t\geq0})$, we can construct the least squares estimators (LSE) and the moment estimators (ME) (See \cite{yu_statistical_2018,xiao_least_2018,xiao2019asymptotic,cai2022mixed} for more details).
\begin{proposition}[\cite{peixingzhi} (4) and (5)]
    The estimator of $\mu$ is the continuous-time sample mean:
    \begin{equation}
        \hat{\mu}=\frac{1}{T}\int_{0}^{T}V_{t}\dif t.
        \label{mu_moment}
    \end{equation}
    The second moment estimator of $k$ is given by
    \begin{equation}
        \hat{k}=\biggl[ \frac{\frac{1}{T}\int_{0}^{T}V_{t}^{2}\dif t-(\frac{1}{T}\int_{0}^{T}V_{t}\dif t)^{2}}{C_{\beta }\Gamma (2\beta -1)}\biggr] ^{-\frac{1}{2\beta }}.
        \label{k_moment}
    \end{equation}
\end{proposition}

Following from Xiao and Yu \cite{xiao2019asymptotic}, we present the LSE in  Vasicek model.
\begin{proposition}[\cite{peixingzhi} (7) and (8)]
    The LSE is motivated by the argument of minimize a quadratic function $L(k,\mu)$ of $k$ and $\mu$: 
    \begin{equation*}
        L(k,\mu )=\int_{0}^{T}\bigl(\overset{\cdot }{V_{t}}-k(\mu -V_{t})\bigr)^{2}\dif t,
    \end{equation*}
    Solving the equation, we can obtain the LSE of $k$ and $\mu$, denoted by $\kls$ and $\muls$ respectively.
    \begin{equation}
        \hat{k}_{LS}=\dfrac{V_{T}\int_{0}^{T}V_{t}\dif t-T\int_{0}^{T}V_{t}\dif V_{t}}{T\int_{0}^{T}V_{t}^{2}\dif t-(\int_{0}^{T}V_{t}\dif t)^{2}},
        \label{mu_ls}
    \end{equation}
    \begin{equation}
        \hat{\mu }_{LS}=\frac{V_{T}\int_{0}^{T}V_{t}^{2}\dif t-\int_{0}^{T}V_{t}\dif V_{t}\int_{0}^{T}V_{t}\dif t}{V_{T}\int_{0}^{T}V_{t}\dif t-T\int_{0}^{T}V_{t}\dif V_{t}},
        \label{k_ls}
    \end{equation}
    where the integral $\int_{0}^{T}V_{t}\dif V_{t}$ is an  It\^{o}-Skorohod integral. 
\end{proposition}
Pei \etal \cite{peixingzhi} prove the following consistencies and central limit theorems (CLT) of estimators. 
\begin{theorem}[\cite{peixingzhi}, Theorem 1.2]
    When \Cref{hyp_1} is satisfied, both ME and LSE of  $\mu, k$ are strongly consistent, that is
    \begin{align*}
        &\lim_{T\to\infty} \hat{\mu}=\mu,\;\;\;\;\;\; \lim_{T\to\infty} \hat{k}=k,\;\;\;\mathrm{a.s.};\\
        &\lim_{T\to\infty} \hat{\mu}_{LS}=\mu,\;\;\; \lim_{T\to\infty} \hat{k}_{LS}=k, \;\;\;\mathrm{a.s.}.
    \end{align*}
\end{theorem}
\begin{theorem}[\cite{peixingzhi}, Theorem 1.3]\label{CLT}
    Assume \Cref{hyp_1} is satisfied. When $G_t$ is self-similar and  $\mE[G_1^2] =1$, $\hat{\mu}$ and $\muls$ are asymptotically normal as $T\to\infty$, that is, 
    \begin{equation*}
        T^{1-\beta}(\hat{\mu}-\mu)\stackrel{law}{\longrightarrow} \mathcal{N}(0,1/k^2),\;\;\;\sqrt{T}(\muls-\mu)\distribution \mathcal{N}(0, 1/k^2).
    \end{equation*}
    When $\beta\in(\frac{1}{2},\frac{3}{4})$, 
    \begin{equation*}
        \sqrt{T}(\hat{k}-k)\distribution \mathcal{N}(0, k\sigma^2_\beta/4\beta^2),
    \end{equation*}
    where
    \begin{align}
        \sigma_\beta^2 = (4\beta-1)\biggl[1+\frac{\Gamma(3-4\beta)\Gamma(4\beta-1)}{\Gamma(2\beta)\Gamma(2-2\beta)}\biggr].
        \label{sigma_beta}
    \end{align}
    Simalarly, $\sqrt{T}(\kls-k)$ is also asymptotically normal as $T\to\infty$:
    \begin{align*}
        \sqrt{T}(\kls-k)&\distribution \mathcal{N}(0,k \sigma _{\beta }^{2}).
    \end{align*}
\end{theorem}

We now present the main Theorems for the whole paper, and their details are given in the following sections.
\begin{theorem}\label{thm_km}
    Let $Z$ be a standard Normal random variable, and $\sigma_\beta^2$ be the constant defined by \eqref{sigma_beta}. 
    Assume $\beta\in (1/2,3/4)$ and \Cref{hyp_1} is satisfied. When $T$ is large enough, there exists a constant $C_{\beta,V}$ such that
    \begin{align}
        \sup_{z\in\mathbb{R}}&\abs{\Pro\biggl(\sqrt{\frac{4\beta^2 T}{k\sigma^2_\beta}}(\hat{k}-k)\leq z\biggr)-\Pro(Z\leq z)}\leq \frac{C_{\beta,V}}{T^m},\label{hat_k}\\
        \sup_{z\in\Rnum}&\abs{\Pro\biggl(\sqrt{\frac{T}{k\sigma_\beta^2}}(\kls-k)\biggr)-\Pro(Z\leq z)}\leq \frac{C_{\beta,V}}{T^{3/4-\beta}},\label{hat_kls}
    \end{align}
    where $m= \min\{1/3, (3-4\beta)/2\}$.
\end{theorem}
Next, we show the convergence speed of mean coefficient estimators $\hat{\mu}$ and $\muls$.
\begin{theorem}\label{thm_mum}
    Assume $\beta\in (1/2,1)$, and $G_t$ is a self-similar Gaussian process satisfying \Cref{hyp_1} and $\mE[G_1^2] =1$. Then there exists a constant $C_{\beta,V}$ such that 
    \begin{align}
        \sup_{z\in\mathbb{R}}&\abs{\Pro\biggl(\frac{k}{T^{\beta-1}}(\hat{\mu}-\mu)\leq z\biggr)-\Pro(Z\leq z)}\leq \frac{C_{\beta,V}}{T^{\beta/2}},\label{mumoment}\\
        \sup_{z\in\Rnum}&\abs{\Pro\biggl(\frac{k}{T^{\beta-1}}(\muls-\mu)\biggr)-\Pro(Z\leq z)}\leq \frac{C_{\beta,V}}{T^{(1-\beta)/2}}.\label{thm_muls}
    \end{align}
\end{theorem}

\section{Preliminary}\label{preliminary}
In this section, we recall some basic facts about Malliavin calculus with respect to Gaussian process. The reader is referred to \cite{nourdin_normal_2012,peccati_gaussian_2007,nualart_malliavin_2006} for a more detailed explanation. Let $G=\{G_t,t\in[0,T]\}$ be a continuous centered Gaussian process with $G_0=0$ and covariance function 
\begin{equation}
    \mathbb{E}(G_{t}G_{s})=R(s,t),\,\, s,t\in[0,T],
\end{equation}
defined on a complete probability space $(\Omega,\mathscr{F},\Pro)$, where $\mathscr{F}$ is generated by the Gaussian family $G$.
Denote $\mathcal{E}$ as the the space of all real valued step functions on $[0,T]$. The Hilbert space $\mathfrak{H}$ is
defined as the closure of $\mathcal{E}$ endowed with the inner product:
\begin{equation}
    \langle \mathbbm{1}_{[a,b)},\mathbbm{1}_{[c,d)}\rangle_{\mathfrak{H}}=\mathbb{E}((G_{b}-G_{a})(G_{d}-G_{c})).
    \label{inner_pro}
\end{equation}
We denote $G=\{G(h), h\in{\mathfrak{H}}\}$ as the isonormal Gaussian process on the probability space, indexed by the elements in $\mathfrak{H}$, which satisfies the following isometry relationship:
\begin{equation}
    \mE\bigl(G(h)\bigr)=0,\,\,\mathbb{E}[G(g)G(h)]=\langle g,h\rangle_{\mathfrak{H}},\quad \forall\ g,h \in\mathfrak{H}.
    \label{isometry}
\end{equation}

The following Proposition shows the inner products representation of the Hilbert space $\mathfrak{H}$ \cite{jolis_wiener_2007}.
\begin{proposition}[\cite{chen_parameter_2021} Proposition 2.1]
    Denote $\mathcal{V}_{[0,T]}$ as the set of bounded variation functions on $[0,T]$. Then $\mathcal{V}_{[0,T]}$ is dense in $\mathfrak{H}$ and
    \begin{equation*}
        \langle f,g\rangle_\mathfrak{H}=\int_{[0,T]^{2}}R(t,s){v}_{f}(\dif t){v}_{g}(\dif s),  \quad    \forall\ f,g\in \mathcal{V}_{[0,T]},
    \end{equation*}
    where $v_g$ is the Lebesgue-Stieljes signed measure associated with $g^0$ defined as
    \begin{equation*}
        g^{0}= \left\{
        \begin{array}{rcr}
            g(x),&& {\text{if}\ x\in [0,T];}\\
            0,&& \text{otherwise}.\\
        \end{array}
    \right.
    \end{equation*}
    When the covariance function $R(t,s)$ satisfies \Cref{hyp_1}, 
    \begin{align}
        \langle f,g\rangle_\mathfrak{H}=\int_{[0,T]^{2}} f(t)g(s)\frac{\partial ^{2}R(t,s)}{\partial t\partial s}\dif t\dif s,  \quad  \forall\ f,g\in \mathcal{V}_{[0,T]}.
        \label{inner product 000}
    \end{align}
    Furthermore, the norm $\left\Vert\cdot\right\Vert_{\mathfrak{H}}$ of the elements in $\mathfrak{H}$ can be induced naturally:
    \begin{equation*}
        \left\Vert\phi\right\Vert_{\mathfrak{H}}^{2}= \int_{[0,T]^{2}}\phi(r_{1})\phi(r_{2})\frac{\partial ^{2}R(r_1,r_2)}{\partial r_1\partial r_2}\dif r_{1}\dif r_{2},\;\;\;\forall\ \phi\in\mathfrak{H}.
    \end{equation*}
\end{proposition}
\begin{remarkn}[\cite{chen_parameter_2021} Notation 1]
    Let $C_{\beta}$ and $C_{\beta}^{\prime}$ be the constants given in \Cref{hyp_1}. For any $\phi(r)\in\mathcal{V}_{[0,T]}$, we define two norms as
    \begin{align*}
        \norm{\phi}_{\mathfrak{H}_1}^{2}&= C_{\beta}\int_{[0,T]^{2}}\phi(r_1)\phi(r_2)\abs{r_1-r_2}^{2\beta-2}\dif r_1\dif r_2,\\
        \norm{\phi}_{\mathfrak{H}_2}^{2}&= C_{\beta}^{\prime}\int_{[0,T]^{2}}\abs{\phi(r_1)\phi(r_2)}(r_1 r_2)^{\beta-1}\dif r_1\dif r_2.
    \end{align*}
    For any $\varphi(r,s)$ in $[0,T]^2$, define an operator $K$ from $\mathcal{V}_{[0,T]}^{\otimes 2}$ to $\mathcal{V}_{[0,T]}$ to be 
    \begin{equation}
        (K\varphi)(r)=\int_{0}^{T}\abs{\varphi(r,u)}u^{\beta-1}\dif u.
    \end{equation}
\end{remarkn}
\begin{proposition}[\cite{chen_parameter_2021} Proposition 3.2]\label{prop_norm_ieq}
    Suppose that \Cref{hyp_1} holds, then for any $\phi(r)\in\mathcal{V}_{[0,T]}$,
    \begin{equation}
        \abs{\norm{\phi}_{\mathfrak{H}}^{2}-\norm{\phi}_{\mathfrak{H}_{1}}^{2}}\leq \norm{\phi}_{\mathfrak{H}_{2}}^{2},
        \label{2.9}
    \end{equation}
    and for any $\varphi,\psi \in (\mathcal {V}_{[0,T]})^{\odot 2}$,
    \begin{align*}
        \abs{\norm{\phi}_{\mathfrak{H}^{\otimes{2}}}^{2}-\norm{\phi}_{\mathfrak{H}_{1}^{\otimes{2}}}^{2}}& \leq \norm{\phi}_{\mathfrak{H}_{2}^{\otimes{2}}}^{2}+2C_{\beta}^{\prime}\norm{K\varphi}_{\mathfrak{H}_{1}}^{2},\\
        \abs{\langle\varphi,\psi\rangle_{\mathfrak{H}^{\otimes{2}}}-\langle\varphi,\psi\rangle_{\mathfrak{H}_{1}^{\otimes{2}}}}& \leq \abs{\langle\varphi,\psi\rangle_{\mathfrak{H}_{2}^{\otimes{2}}}}+2C_{\beta}^{\prime}\abs{\langle K\varphi,K\psi\rangle_{\mathfrak{H}_{1}}}.
    \end{align*}
\end{proposition}
Let $\mathfrak{H}^{\otimes p}$ and $\mathfrak{H}^{\odot p}$ be the $p$-th tensor product and the $p$-th symmetric tensor product of $\mathfrak{H}$. For every $p\geq 1$, denote $\mathcal{H}_p$ as the $p$-th Wiener chaos of $G$. It is defined as the closed linear subspace of $L^2(\Omega)$ generated by $\{H_{p}(G(h)): h\in{\mathfrak{H}},\normh{h}=1\}$, where $H_p$ is the $p$-th Hermite polynomial. Let $h\in \mathfrak{H}$ such that $\normh{h}=1$, then for every $p\geq 1$ and $h\in\mathfrak{H}$, 
\begin{equation*}
    I_{p}(h^{\otimes p})= H_{p}\bigl(G(h)\bigr),
\end{equation*}
where $I_p(\cdot)$ is the $p$-th Wiener-It\^{o} stochastic integral.

Denote $\{e_k,k\geq 1\}$ as a complete orthonormal system in $\mathfrak{H}$. The $q$-th contraction between $f\in\mathfrak{H}^{\odot m}$ and $g\in\mathfrak{H}^{\odot n}$ is an element in $\mathfrak{H}^{\otimes (m+n-2q)}$: 
\begin{align*}
    f \otimes_q g =\sum_{i_1,\cdots,i_q=1}^{\infty}\langle f,e_{i_1},\cdots,e_{i_q}\rangle_{\mathfrak{H}^{\otimes q}}\otimes \langle g,e_{i_1},\cdots,e_{i_q}\rangle_{\mathfrak{H}^{\otimes q}},\;\;\; q= 1,\cdots,m\wedge n.
\end{align*}
The following proposition shows the product formula for the multiple integrals. 
\begin{proposition}[\cite{nourdin_normal_2012} Theorem 2.7.10]\label{lemma_mutli_equa}
    Let $f\in\mathfrak{H}^{\odot p}$ and $g\in\mathfrak{H}^{\odot q}$ be two symmetric function. Then
    \begin{equation}
        I_{p}(f)I_{q}(g)=\sum_{r=0}^{p\wedge q}r!\tbinom{p}{r}\tbinom{q}{r}I_{p+q-2r}(f\tilde{\otimes}_{r}g),
        \label{mutil_n}
    \end{equation} 
    where $f\tilde{\otimes}_{r}g$ is the symmetrization of $f {\otimes}_{r}g$.
\end{proposition}
We then introduce the derivative operator and the divergence operator. For these details, see sections 2.3-2.5 of  \cite{nourdin_normal_2012}. 
Let $\mathscr{T}$ be the class of smooth random variables of the form:
\begin{equation*}
    F= f\bigl(G(\psi_1),G(\psi_2),\cdots, G(\psi_n)\bigr),
\end{equation*}
where $n\geq 1$, $f\in\mathcal{C}_b^\infty(\mathbb{R}^n)$ which partial derivatives have at most polynomial growth, and for $i=1,\cdots,n$, $\psi_i\in \mathfrak{H}$. Then, the Malliavin derivative of $F$ (with respect to $G$) is the element of $L^2(\Omega,\mathfrak{H})$ defined by 
\begin{equation*}
    D F = \sum_{i_=1}^n \frac{\partial f}{\partial x_i} \bigl(G(\psi_1),\cdots, G(\psi_n)\bigr)\psi_i.
\end{equation*}
Given $q\in[1,\infty)$ and integer $p\geq 1$, let $\mathbb{D}^{p,q}$ denote the closure of $\mathscr{T}$ with respect to the norm
\begin{equation*}
    \norm{F}_{\mathbb{D}^{p,q}} = \Bigl[\mE(\abs{F}^q)+ \sum_{k=1}^p \mE(\norm{D^k F}^q_{\mathfrak{H}\otimes k})\Bigr]^{1/q}.
\end{equation*}
Denote $\delta$ (the divergence operator) as the adjoint of $D$. The domain of $\delta$ is composed of those elements:
\begin{equation*}
    \abs{\mE[\langle D^p F, u\rangle_{\mathfrak{H}\otimes p}]}\leq C[\mE(\abs{F}^2)]^{\frac{1}{2}},\;\;\; \forall\ F\in \mathbb{D}^{p,2},
\end{equation*}
and is denoted by $\mathrm{Dom}(\delta)$. If $u\in \mathrm{Dom}(\delta)$, then $\delta(u)$ is the unique element of  $L^2(\Omega)$ characterized by the duality formula:
\begin{equation*}
    \mE [F\delta (u)] =\mE (\langle D F, u\rangle_{\mathfrak{H} }), \;\;\; \forall\ F\in \mathbb{D}^{1,2}.
\end{equation*}

We now introduce the infinitesimal generator $L$ of the Ornstein-Uhlenbeck semigroup. Let $F\in L^2(\Omega)$ be a square integrable random variable. Denote $\mathbb{J}_n : L^2(\Omega)\to \mathcal{H}_n$ as the orthogonal projection on the $n$-th Wiener chaos $\mathcal{H}_n$. The operator $L$ is defined by $L F=- \sum_{n=0}^\infty n \mathbb{J}_n F$. The domain of $L$ is 
\begin{align*}
    \mathrm{Dom} (L) := \{F\in L^2(\Omega), \sum_{p=1}^\infty p^2 \mE [\mathbb{J}_p(F)^2]\leq \infty\}.
\end{align*}
For any $F\in L^2(\Omega)$, define $L^{-1} F =- \sum_{n=1}^\infty \frac{1}{n}\mathbb{J}_n F$. $L^{-1}$ is called the Pseudo-inverse of $L$. Note that $L^{-1} F\in\mathrm{Dom} (L)$ and $L L^{-1} F= F - \mE (F)$ holds for any $F\in L^2(\Omega)$.

The following \Cref{chang_lemma} provides the \bE upper bound on the sum of two random variables.

\begin{lemma}[\cite{chang1989berry} Lemma 2]\label{chang_lemma}
    For any variable $\xi,\eta$ and $a\in\Rnum$, the following inequality holds:
    \begin{equation}
        \sup_{z\in\mathbb{R}}\abs{\Pro(\xi+\eta\leq z)-\Phi(z)}\leq \sup_{z\in\mathbb{R}}\abs{\Pro(\xi\leq z)-\Phi(z)}+\Pro(\abs{\eta}>a)+\frac{a}{\sqrt{2\pi}},
        \label{chang_ieq}
    \end{equation}
    where $\Phi(z)$ is the standard Normal distribution function. 
\end{lemma}

Using Malliavin calculus, Kim and Park \cite{kim2017optimal} provide the \bE  upper bound of the quotient of two random variables.

Let $F_T\in\mathbb{D}^{1,2}$ be a zero-mean process, and $G_T\in\mathbb{D}^{1,2}$ satisfies $G_T >0$ a.s.. For simplicity, we define the following four functions:
\begin{align*}
    \Psi_1(T) &= \frac{1}{(\mE{G_T})^2} \sqrt{\mE\Bigl[\Bigl((\mE G_T)^2-(\langle DF_T ,- DL^{-1} F_T\rangle_{\mathfrak{H}})\Bigr)^2\Bigr]},\\
    \Psi_2(T) &= \frac{1}{(\mE{G_T})^2} \sqrt{\mE\Bigl[\langle DF_T,-DL^{-1}(G_T-\mE G_T)\rangle_{\mathfrak{H}}^2\Bigr]},\\
    \Psi_3(T) &= \frac{1}{(\mE{G_T})^2} \sqrt{\mE\Bigl[\langle DG_T,-DL^{-1}F_T\rangle_{\mathfrak{H}}^2\Bigr]},\\
    \Psi_4(T) &= \frac{1}{(\mE{G_T})^2} \sqrt{\mE \Bigl[\langle DG_T ,-DL^{-1}(G_T-\mE G_T)\rangle_{\mathfrak{H}}^2\Bigr]}.
\end{align*}

\begin{theorem}[\cite{kim2017optimal} Theorem 2 and Corollary 1]\label{kim2017}
    Let $Z$ be a standard Normal variable. Assuming that for every $z\in\Rnum$, $F_T+zG_T$ has an absolutely continuous law with respect to Lebesgue measure and $\Psi_i(T)\to 0$, $i=1,\cdots,4$, as $T\to\infty$. Then, there exists a constant $c$ such that for $T$ large enough,
    \begin{align*}
        \sup_{z\in\Rnum} \abs{\Pro \Bigl(\frac{F_T}{G_T}\leq z\Bigr)-\Pro(Z\leq z)}&\leq c\cdot \max_{i=1,\cdots,4} \Psi_i(T).
    \end{align*}
\end{theorem}


\section{Berry-Ess\'{e}en upper bounds of moment estimators}\label{sec_moment}

In this section, we will prove the \bE upper bounds of Vasicek model moment estimators $\hat{\mu}$ and $\hat{k}$.
For the convenience of the following discussion, we first define $A(z)$:
\begin{align}
    A(z):=\Pro\biggl(\frac{k}{T^{\beta-1}}(\hat{\mu}-\mu)\leq z\biggr)-\Pro(Z\leq z),
    \label{formal_az}
\end{align}
where $Z\sim \mathcal{N}(0,1)$ is a standard Normal variable. Next, we introduce the CLT of $\hat{\mu}$.
\begin{theorem}[\cite{peixingzhi} Proposition 4.19]
    Assume $\beta\in (1/2,1)$, and $G_t$ is a self-similar Gaussian process satisfying \Cref{hyp_1} and $\mE[G_1^2] =1$. Then $T^{1-\beta}(\hat{\mu}-\mu)$ is asymptotically normal as $T\to\infty$:
    \begin{align}
        T^{1-\beta}(\hat{\mu}-\mu)=\frac{\mu}{k}\frac{\me^{-kT}-1}{T^\beta}+\frac{1}{k}\frac{G_T-Z_T}{T^\beta}\distribution \mathcal{N}(0,1/k^2),
    \end{align}
    where
    \begin{equation*}
        Z_T = I_1\bigl(\me^{-k(T-s)}\mathbbm{1}_{[0,T]}(s)\bigr)
    \end{equation*}
    is stochastic integral with respect to $G_t$.
\end{theorem}
Following from the above Theorem, we can obtain the expanded form of \eqref{formal_az}:
\begin{align*}
    A(z)&=\Pro\Bigl(kT^{1-\beta}(\hat{\mu}-\mu)\leq z\Bigr)-\Pro(Z\leq z)\\
    &=\Pro\biggl(\frac{G_T-Z_T+\mu(\me^{-kT}-1)}{T^\beta}\leq z\biggr)-\Pro(Z\leq z).
\end{align*}
Then, we can prove the convergence speed of $\hat{\mu}$.
\begin{proof}[Proof of formula \eqref{mumoment}]
    Let $a=T^{-\beta/2}$, According to \Cref{chang_lemma}, we have
    \begin{align*}
        \sup_{z\in\mathbb{R}}\abs{A(z)}&\leq \sup_{z\in\mathbb{R}}\abs{\Pro\Bigl(\frac{G_T}{T^\beta}\leq z\Bigr)-\Pro(Z\leq z)}\\
        &\ \ \ +\Pro\biggl(\abs{\frac{-Z_T+\mu(\me^{-kT}-1)}{T^\beta}}>T^{-\frac{\beta}{2}}\biggr)+\frac{T^{-\frac{\beta}{2}}}{\sqrt{2\pi}}.
    \end{align*}
    Since $G_T$ is self-similar, $(G_T/{T^\beta})$ is standard Normal variable,
    \begin{align*}
        \sup_{z\in\mathbb{R}}\abs{\Pro\Bigl(\frac{G_T}{T^\beta}\leq z\Bigr)-\Pro(Z\leq z)}=0.
    \end{align*}
    Following from Chebyshev inequality, we can obtain
    \begin{align*}
        \Pro\biggl(\abs{\frac{-Z_T+\mu(\me^{-kT}-1)}{T^\beta}}>T^{-\frac{\beta}{2}}\biggr)=\Pro\bigl(\abs{Z_T-\mu(\me^{-kT}-1)}>T^{\frac{\beta}{2}}\bigr)\leq \frac{C_1}{T^{\beta}},
    \end{align*}
    where 
    \begin{align*}
        C_1 = \mE\bigl(\abs{Z_T-\mu(\me^{-kT}-1)}^2\bigr).
    \end{align*}
    The Proposition 3.10 of \cite{peixingzhi} ensures that $C_1$ is bounded. Combining the above results, we have
    \begin{align}
        \sup_{z\in\mathbb{R}}\abs{A(z)}\leq \frac{1}{\sqrt{2\pi}T^{\beta/2}}+\frac{C_1}{T^{\beta}}.
    \end{align}
    When $T$ is sufficiently large, there exists the constant $C_{\beta,V}$ such that the formula \eqref{mumoment} holds.
\end{proof}
Similarly, we review the central limit theorem of $\hat{k}$.
\begin{theorem}[\cite{peixingzhi} Proposition 4.18]
    Assume $\beta\in(1/2,3/4)$ and $G_t$ is a Gaussian process satisfying \Cref{hyp_1}. Then $\sqrt{T}(\hat{k}-k)$ is asymptotically normal as $T\to\infty$:
    \begin{align*}
        \sqrt{T}(\hat{k}-k) = \sqrt{T}\left(\left[ \frac{\frac{1}{T}\int_{0}^{T}V_{t}^{2}dt-(\frac{1}{T}\int_{0}^{T}V_{t}dt)^{2}}{C_{\beta }\Gamma (2\beta -1)}\right] ^{-\frac{1}{2\beta }}-k\right)\distribution \mathcal{N}(0,k\sigma^2_\beta/4\beta^2).
    \end{align*}
\end{theorem}
The following Lemma shows the upper bound of the expectation of $(\int_0^T \me^{-ks}\dif G_s)^2$.
\begin{lemma}\label{Lemma_MT}
    Let $M_T$ be the process defined by
    \begin{equation*}
        M_T=I_1\bigl(\me^{-ku}\mathbbm{1}_{[0,T]}(u)\bigr)=\int_0^T \me^{-ks}\dif G_s. 
    \end{equation*}
    When $\beta\in(1/2,1)$, there exists constant $C$ independent of $T$ such that
    \begin{equation}
        \mE(M^2_T)\leq C.
        \label{formula_MT}
    \end{equation}
\end{lemma}
\begin{proof}
    According to \eqref{isometry} and \eqref{2.9}, we can obtain
    \begin{align*}
        \mE[\abs{M_T}^2]=\normh{m_T}^2 \leq \norm{m_T}^2_{\mathfrak{H}_1}+\norm{m_T}^2_{\mathfrak{H}_2},
    \end{align*}
    where $m_T(u)=\me^{-ku}\mathbbm{1}_{[0,T]}(u)$.

    It is easy to see that 
    \begin{align}
        \norm{m_T}^2_{\mathfrak{H}_1} &=2C_\beta\int_{0\leq u\leq v\leq T} \me^{-k(u+v)}\abs{u-v}^{2\beta-2}\dif u\dif v\notag\\
        &\leq C\int_{0}^T(\int_0^T \me^{-k(2x+y)}x^{2\beta-2}\dif x)\dif y\leq C^\prime,
    \end{align}
    where $C^\prime$ is a constant. Also, we have 
    \begin{align}
        \norm{m_T}^2_{\mathfrak{H}_2}&=C^\prime_\beta\int_0^T\me^{-ku}u^{\beta-1}\dif u\int_0^T\me^{-kv}v^{\beta-1}\dif v\leq C^{\prime\prime}.
    \end{align}
    Combining the above two formulas, we obtain \eqref{formula_MT}.
\end{proof}

Denote $B(z)$ as
\begin{align*}
    B(z):&=\Pro\biggl(\sqrt{\frac{4\beta^2 T}{k\sigma^2_\beta}}(\hat{k}-k)\leq z\biggr)-\Pro(Z\leq z).
\end{align*}
Then we can obtain the Berry-Ess\'{e}en upper bound of ME $\hat{k}$.
\begin{proof}[Proof of formula \eqref{hat_k}]
    According to \cite{peixingzhi} Proposition 4.18, we have
        \begin{align*}
            B(z)&=\Pro\biggl(\sqrt{\frac{4\beta^2 T}{k\sigma^2_\beta}}(\hat{k}-k)\leq z\biggr)-\Pro(Z\leq z)\\
            &=\Pro\biggl(\hat{k}-k\leq \sqrt{\frac{k\sigma^2_\beta}{4\beta^2 T}}z\biggr)-\Pro(Z\leq z)\\
            &=\Pro\biggl(\frac{\frac{1}{T}\int_0^T V^2_t\dif t - \bigl(\frac{1}{T}\int_0^T V_t\dif t\bigr)^2}{C_\beta \Gamma(2\beta -1)}\geq \biggl(\sqrt{\frac{k\sigma^2_\beta}{4\beta^2 T}}z+k\biggr)^{-2\beta}\biggr)-\Pro(Z\leq z)\\
            &=\Pro\biggl(\frac{1}{T}\bigl(\int_0^T V^2_t\dif t -(\frac{1}{T}\int_0^T V_t\dif t)^2\bigr)-\alpha\\
            &\; \; \;\;\;\;\;\;\; \geq C_\beta \Gamma(2\beta -1)\biggl[\biggl(\sqrt{\frac{k\sigma^2_\beta}{4\beta^2 T}}z+k\biggr)^{-2\beta}-k^{-2\beta}\biggr]\biggr)-\Pro(Z\leq z)\\
            &=\Pro\biggl(\frac{1}{T}\bigl(\int_0^T V^2_t\dif t -(\frac{1}{T}\int_0^T V_t\dif t)^2\bigr)-\alpha\geq \alpha\biggl[\biggl(1+\frac{z\sigma_\beta}{2\beta\sqrt{k T}}\biggr)^{-2\beta}-1\biggr]\biggr)\\
            &\ \ \ -\Pro(Z\leq z).
        \end{align*}
        We denote $\overline{\Phi}(z)$ as the tail probability $1-\Pro(Z\leq z)$ and 
        \begin{equation*}
            \nu =\sqrt{\frac{kT}{\sigma_\beta^2}}\biggl[\biggl(1+\frac{z\sigma_\beta}{2\beta\sqrt{k T}}\biggr)^{-2\beta}-1\biggr].
        \end{equation*}
        Then we can obtain
        \begin{align}
            \abs{B(z)}&=\abs{\Pro\biggl(\sqrt{\frac{kT}{\alpha^2\sigma_\beta^2}}\biggl[\frac{1}{T}\bigl(\int_0^T V^2_t\dif t -(\frac{1}{T}\int_0^T V_t\dif t)^2\bigr)-\alpha\biggr]\geq\nu\biggr)-\Pro(Z\leq z)}\notag\\
            &\leq \abs{\Pro\biggl(\sqrt{\frac{kT}{\alpha^2\sigma_\beta^2}}\biggl[\frac{1}{T}\bigl(\int_0^T V^2_t\dif t -(\frac{1}{T}\int_0^T V_t\dif t)^2\bigr)-\alpha\biggr]\geq\nu\biggr)-\overline{\Phi}(\nu)}\notag\\
            &\ \ \ +\abs{\overline{\Phi}(\nu)-\Pro(Z\leq z)} \notag\\
            &= \abs{D(\nu)}+\abs{\overline{\Phi}(\nu)-\Pro(Z\leq z)}.
            \label{kmoment_step1}
    \end{align}
    Denote $D(v)$ as 
    \begin{align*}
        D(\nu):&=\abs{\Pro\biggl(\sqrt{\frac{kT}{\alpha^2\sigma_\beta^2}}\biggl[\frac{1}{T}\bigl(\int_0^T V^2_t\dif t -(\frac{1}{T}\int_0^T V_t\dif t)^2\bigr)-\alpha\biggr]\geq\nu\biggr)-\overline{\Phi}(\nu)},
    \end{align*}
    where $\nu\in\Rnum$, $\alpha:=C_\beta \Gamma(2\beta-1)k^{-2\beta}$. The Lemma 5.4 of \cite{chen_parameter_2021}  ensures that 
    \begin{equation*}
        \abs{\overline{\Phi}(\nu)-\Pro(Z\leq z)}\leq \frac{C}{\sqrt{T}}.
    \end{equation*}
    Combining with \Cref{lemma_dv}, we obtain the desired result.
\end{proof}
The following Lemma provides the upper bound of $D(\nu)$.
\begin{lemma}\label{lemma_dv}
    When $T$ is large enough, there exists constant $C_{\beta,V}^\prime$ such that
    \begin{align*}
        \sup_{\nu\in\Rnum}\abs{D(\nu)}\leq \frac{C_{\beta,V}^\prime}{T^m},
    \end{align*}
    where $m= \min\ \{1/3, (3-4\beta)/2\}$.
\end{lemma}
\begin{proof}
    Since the Normal distribution is symmetric, we have
    \begin{align*}
        D(\nu)&=\abs{\Pro\biggl(\sqrt{\frac{kT}{\alpha^2\sigma_\beta^2}}\biggl[\frac{1}{T}\bigl(\int_0^T V^2_t\dif t -(\frac{1}{T}\int_0^T V_t\dif t)^2\bigr)-\alpha\biggr]\geq\nu\biggr)-\overline{\Phi}(\nu)}\\
        &=\abs{\Pro\biggl(\sqrt{\frac{kT}{\alpha^2\sigma_\beta^2}}\biggl[\frac{1}{T}\bigl(\int_0^T V^2_t\dif t -(\frac{1}{T}\int_0^T V_t\dif t)^2\bigr)-\alpha\biggr]\leq\nu\biggr)-\Phi(\nu)}.
    \end{align*}
    Consider the following processes:
    \begin{align}
        X_t &= \int_0^t \me^{-k(t-s)}\dif G_s,\;\;\;I_5 = \frac{1}{T}\biggl(\int_0^T X_t^2\dif t\biggr),
        \label{formal_OU}\\
        F_T &= \int_0^T \int_0^t \me^{-k(t-s)}\dif G_s \dif t,\notag
    \end{align}
    where $X_t$ is an OU process driven by $G_t$. According to \cite{peixingzhi} formula (63), we can obtain
    \begin{align*}
        \frac{1}{T}\int_0^T V_t^2\dif t -\biggl(\frac{1}{T} \int_0^T V_t\dif t\biggr)^2 -\alpha = \frac{1}{T}\biggl(\int_0^T X_t^2\dif t\biggr) -\alpha +K_T+I_4,
    \end{align*}
    where 
    \begin{align*}
        K_T&=\frac{\mu^2}{2k}\cdot \frac{1-\me^{-2kT}}{T}-\frac{\mu^2(\me^{-kT}-1)^2}{k^2 T^2},\\
        I_4&=\frac{\mu}{k}\biggl[\frac{Z_T}{T}\me^{-kT}-\frac{M_T}{T}+2(1-\me^{-kT})\frac{F_T}{T^2}\biggr]-\frac{F_T^2}{T^2}.
    \end{align*}
    Let $a=T^{-1/3}$. \Cref{chang_lemma} ensures that
    \begin{equation}
        \begin{split}
            \sup_{\nu\in\Rnum}\abs{D(\nu)}&\leq \sup_{\nu\in\Rnum}\abs{{\Pro\biggl(\sqrt{\frac{kT}{\alpha^2\sigma_\beta^2}}\bigl(I_5-\alpha+K_T\bigr)\leq \nu\biggr)}-\Pro(Z\leq \nu)}\\
            &\ \ \ +\Pro\biggl(\abs{\sqrt{\frac{kT}{\alpha^2\sigma_\beta^2}}I_4}>T^{-1/3}\biggr)+\frac{T^{-1/3}}{\sqrt{2\pi}}.
        \end{split}
        \label{3.3}
    \end{equation}
    According to \cite{chen_parameter_2021} Theorem 1.4, we have 
    \begin{equation}
        \sup_{\nu\in\Rnum}\abs{{\Pro\biggl(\sqrt{\frac{kT}{\alpha^2\sigma_\beta^2}}\bigl(I_5-\alpha+K_T\bigr)\leq \nu\biggr)}-\Pro(Z\leq \nu)} \leq \frac{C_{k,\beta}}{T^{\frac{3-4\beta}{2}}},
        \label{firstterm}
    \end{equation}
    where $C_{k,\beta}$ independent of $T$ is a constant. Denote $K_1=4\sqrt{\frac{k}{\alpha^2\sigma_\beta^2}}$. We then consider the second term of right side of \eqref{3.3}.
    \begin{align*}
        \Pro\biggl(\abs{\sqrt{\frac{kT}{\alpha^2\sigma_\beta^2}}I_4}>T^{-1/3}\biggr)&\leq \Pro\biggl(\abs{K_1\frac{\mu\cdot\me^{-kT}\cdot Z_T}{k}}>T^{1/6}\biggr)\\
            &\ \ \ +\Pro\biggl(\abs{K_1\frac{\mu\cdot M_T}{k}}>T^{1/6}\biggr)\\
            &\ \ \ +\Pro\biggl(\abs{K_1\frac{\mu\cdot(2-2\me^{-kT})\cdot F_T}{k}}>T^{7/6}\biggr)\\
            &\ \ \ +\Pro\biggl(\abs{K_1 \cdot F^2_T}>T^{7/6}\biggr).
    \end{align*}
    Combining the Chebyshev inequality, \Cref{Lemma_MT} and the Proposition 3.10 of \cite{peixingzhi}, we can obtain 
    \begin{align*}
        \Pro\biggl(\abs{K_1\frac{\mu\cdot\me^{-kT}\cdot Z_T}{k}}>T^{1/6}\biggr)&\leq \frac{C^\prime_1\mE(Z_T^2)}{T^{1/3}}\leq \frac{C_1}{T^{1/3}},\\
        \Pro\biggl(\abs{K_1\frac{\mu\cdot M_T}{k}}>T^{1/6}\biggr)&\leq \frac{C^{\prime}_{2}\mE(M_T^2)}{T^{1/3}}\leq \frac{C_2}{T^{1/3}},\\
        \Pro\biggl(\abs{K_1\frac{\mu\cdot(2-2\me^{-kT})\cdot F_T}{k}}>T^{13/6}\biggr)&\leq \frac{C^\prime_3\mE(F_T^2)}{T^{13/3}}\leq \frac{C_3}{T^{13/3-2\beta}},\\
        \Pro\Bigl(\abs{K_1 \cdot F^2_T}>T^{13/6}\Bigr)&\leq \frac{C^\prime_4\mE(F_T^2)}{T^{13/6}}\leq \frac{C_4}{T^{13/6-2\beta}}.
    \end{align*}
    Then we have 
    \begin{equation}
        \Pro\biggl(\abs{\sqrt{\frac{kT}{a^2\sigma_\beta^2}}I_4}>T^{-1/3}\biggr)+\frac{T^{-1/3}}{\sqrt{2\pi}}\leq \frac{C_{k,\beta}^\prime}{T^{1/3}},
        \label{secondterm}
    \end{equation}
    where $C_{k,\beta}^\prime$ is a constant. Combining formulas \eqref{firstterm} and \eqref{secondterm}, we obtain the desired result.
\end{proof}
\section{Berry-Ess\'{e}en upper bounds of least squares estimators}
For the convenience of following proof, we introduce some variables:
\begin{align*}
    a_T &= 1-\me^{-kT},\;\;\;
    b_T =\frac{1}{T}\int_0^T \norm{\me^{-k(t-\cdot)}\mathbbm{1}_{[0,t]}(\cdot)}_{\mathfrak{H}}^2\dif t\to C_\beta\Gamma(2\beta-1)k^{-2\beta},\\
    c_T &= \int_0^T \mu^2(1-\me^{-kt})^2\dif t = \mu^2(T+\frac{2}{k}(\me^{-kT}-1)+\frac{1}{2k}(1-\me^{-2kT})),\\
    d_T &= \int_0^T (1-\me^{-kt})\dif t=T+\frac{1}{k}(\me^{-kT}-1).
\end{align*}
The Proposition 3.10 of \cite{peixingzhi}  and \Cref{lemma_mutli_equa} ensure that 
\begin{align*}
    e_T &=\norm{l_T\otimes_1 l_T}_{\mathfrak{H}}^2 
    \leq C T^{2\beta},\\
    q_T &=\normh{l_T\otimes_1 k_T}^2 
    \leq \normh{k_T\otimes_1 \frac{1}{k}}^2\leq 
    C T^{2\beta},
\end{align*}
where $C$ is a constant independent of $T$. Also, we show $l_T$ and other functions:
\begin{align*}
    f_T(t,s) &= \me^{-k\abs{t-s}}\mathbbm{1}_{[0,T]^2}(t,s),\; \; 
    & h_T(t,s) = \me^{-k(T-t)-k(T-s)}\mathbbm{1}_{[0,T]^2}(t,s),\\
    g_T(t,s) &= \frac{1}{2kT} f_T-h_T,\\
    k_T(s) &=\me^{-k(T-s)}\mathbbm{1}_{[0,T]}(s),
    & l_T(s) = \frac{1}{k}(1-\me^{-k(T-s)})\mathbbm{1}_{[0,T]}(s),\\
    m_T(s) &= \me^{-ks}\mathbbm{1}_{[0,T]}(s),
    & n_T(s) = \frac{1}{2k}(\me^{-k(2T-s)}-1)\mathbbm{1}_{[0,T]}(s).
\end{align*}
Furthermore, we denote $I_1(f_T)$ as $I_1(f_T(t,\cdot)\mathbbm{1}_{[0,T]}(\cdot))$. 

We now extent the Corollary 1 of \cite{kim2017optimal}.
\begin{theorem}\label{theorem_triple}
     Let $F_T/H_T$ be a zero-mean ratio process that contains at most triple Wiener-It\^{o} stochastic integrals, where $F_T = I_1(f_{1,T})+I_2(f_{2,T})+I_3(f_{3,T})$, $H_T = E_T+I_1(h_{1,T})+I_2(h_{2,T})+I_3(h_{3,T})$ and $E_T$ is a positive function of $T$ converging to a constant $\alpha$. Suppose that $\Psi_i(T)\to 0$, $i=1,\cdots,4$, as $T\to \infty$. Then, there exists constant $C$ such that when $T$ is large enough,
    \begin{equation}
        \begin{split}
            \sup_{z\in\Rnum} \abs{\Pro \Bigl(\frac{F_T}{H_T}\leq z\Bigr)-\Pro(Z\leq z)}&\leq C\cdot \max_{\substack{p,q=f,h\\m,n=1,2,3\\i=1,\cdots,(m\wedge n)}}\Bigl(\norm{p_{m,T}\tilde{\otimes}_i q_{n,T}}_{\mathfrak{H}^{\otimes (m+n-2i)}},\\
            &\;\;\;\iproct{p_{m,T}}{q_{m,T}}_{\mathfrak{H}^{\otimes m}},(E_T^2-\sum_{j=1}^m j!\norm{f_{j,T}}_{\mathfrak{H}^{\otimes j}}^2)\Bigr).  
        \end{split}
    \end{equation}
\end{theorem}
\begin{proof}
    We first consider $\Psi_4(T)=\frac{1}{(\mE H_T)^2}\sqrt{\mE[\iproct{DH_T}{-DL^{-1}(H_T-\mE H_T)}^2_{\mathfrak{H}}]}$. It is easy to see that $(\mE H_T)^2\to E_T^2\to \alpha^2\;\mathrm{a.s.}$.  
    Then we deal with $\sqrt{\mE[\iproct{DH_T}{-DL^{-1}(H_T-\mE H_T)}^2_{\mathfrak{H}}]}$. Denote $\iproct{DH_T}{-DL^{-1}(H_T-\mE H_T)}_{\mathfrak{H}}=\psi_4$, we have 
    \begin{small}
    \begin{align}
        \psi_4&=\iproct{h_{1,T}+2I_1(h_{2,T})+3I_2(h_{3,T})}{h_{1,T}+I_1(h_{2,T})+I_2(h_{3,T})}_{\mathfrak{H}}\notag\\
        &=\normh{h_{1,T}}^2+3h_{1,T} I_1(h_{2,T})+4h_{1,T} I_2(h_{3,T})\notag\\
        &\;\;\;+2\norm{h_{2,T}}_{\mathfrak{H}^{\otimes 2}}^2+2I_2(h_{2,T}\tilde{\otimes}_1 h_{2,T})\label{psi4}\\
        &\;\;\;+5I_3(h_{2,T}\tilde{\otimes}_1 h_{3,T})+10I_1(h_{2,T}\tilde{\otimes}_2 h_{3,T})\notag\\
        &\;\;\;+6\norm{h_{3,T}}_{\mathfrak{H}^{\otimes 2}}^2+3I_4(h_{3,T}\tilde{\otimes}_1 h_{3,T})+12I_2(h_{3,T}\tilde{\otimes}_2 h_{3,T}).\notag
    \end{align}
    \end{small}
    Following from the the orthogonality property of multiple integrals, we can obtain
    \begin{align}
        \Psi_4(T)&\leq C_1\cdot \max_{\substack{m,n=1,2,3\\i=1,\cdots,(m\wedge n)}}\Bigl(\norm{h_{m,T}\tilde{\otimes}_i h_{n,T}}_{\mathfrak{H}^{\otimes (m+n-2i)}},\norm{h_{m,T}}_{\mathfrak{H}^{\otimes m}}^2\Bigr),
        \label{bound_4}
    \end{align}
    where $C_1$ is a constant independent of $T$.

    We next consider $\Psi_2(T)$ and $\Psi_3(T)$. Denote $\iproct{DF_T}{-DL^{-1}(H_T-\mE H_T)}_{\mathfrak{H}}=\psi_2$, we have 
    \begin{align*}
        \psi_2&=\iproct{f_{1,T}+2I_1(f_{2,T})+3I_2(f_{3,T})}{h_{1,T}+I_1(h_{2,T})+I_2(h_{3,T})}_{\mathfrak{H}}\\
        &=f_{1,T}h_{1,T}+f_{1,T} I_1(h_{2,T})+f_{1,T} I_2(h_{3,T})\\
        &\;\;\;+2h_{1,T}I_1(f_{2,T})+3h_{1,T}I_2(f_{3,T})\\
        &\;\;\;+2\iproct{f_{2,T}}{h_{2,T}}_{\mathfrak{H}^{\otimes 2}}+2I_2(f_{2,T}\tilde{\otimes}_1 h_{2,T})\\
        &\;\;\;+2I_3(f_{2,T}\tilde{\otimes}_1 h_{3,T})+4I_1(f_{2,T}\tilde{\otimes}_2 h_{3,T})\\
        &\;\;\;+3I_3(f_{3,T}\tilde{\otimes}_1 h_{2,T})+6I_1(f_{3,T}\tilde{\otimes}_2 h_{2,T})\\
        &\;\;\;+6\iproct{f_{3,T}}{h_{3,T}}_{\mathfrak{H}^{\otimes 3}}+3I_4(f_{3,T}\tilde{\otimes}_1 h_{3,T})+12I_2(f_{3,T}\tilde{\otimes}_2 h_{3,T}).
    \end{align*}
    Simalarly, there exists a constant $C_2$ such that  
    \begin{align}
        \Psi_2(T)&\leq C_2\cdot \max_{\substack{p,q=f,h\\m,n=1,2,3\\i=1,\cdots,(m\wedge n)}}\Bigl(\norm{p_{m,T}\tilde{\otimes}_i q_{n,T}}_{\mathfrak{H}^{\otimes (m+n-2i)}},\iproct{f_{m,T}}{h_{m,T}}_{\mathfrak{H}^{\otimes m}}\Bigr).
        \label{bound_2}
    \end{align}
    Following from the above result, we can obtain the bound of $\Psi_3(T)$.

    We then deal with $\Psi_1(T)= \frac{1}{(\mE{H_T})^2} \sqrt{\mE[((\mE H_T)^2- \iproct{DF_T}{-DL^{-1} F_T}_{\mathfrak{H}})^2]}$. According to the orthogonality and \eqref{psi4}, we have
    \begin{align}
        \Psi_1(T)&\leq C_3\cdot \max_{\substack{m,n=1,2,3\\i=1,\cdots,(m\wedge n)}}\Bigl(\norm{f_{m,T}\tilde{\otimes}_i f_{n,T}}_{\mathfrak{H}^{\otimes (m+n-2i)}},(E_T^2-\sum_{i=1}^m j!\norm{f_{j,T}}_{\mathfrak{H}^{\otimes j}}^2)\Bigr).
        \label{bound_1}
    \end{align}
    Combining formulas \eqref{bound_4}, \eqref{bound_2} and \eqref{bound_1}, we obtain the desired result.
\end{proof}

We now prove the convergence speed of $\kls$. First, we review the CLT of $\kls$.
\begin{theorem}[\cite{peixingzhi} Propositions 4.20]
    When $\beta\in(1/2,3/4)$ and \Cref{hyp_1} holds, $\kls$ satisfies the following central limit theorem:
    \begin{align}
        \sqrt{T}(\kls - k) &= \frac{\frac{V_T}{\sqrt{T}}\hat{\mu}+k\sqrt{T}\hat{\mu}(\hat{\mu}-\mu)-\frac{1}{\sqrt{T}}\int_0^T V_t\dif G_t}{\frac{1}{T}\int_0^T V_t^2 \dif t - (\frac{1}{T}\int_0^T V_t\dif t)^2}\distribution \mathcal{N}(0,k\sigma_\beta^2).
            \label{klsclt}
    \end{align}
\end{theorem}
We then transfrom the above $\sqrt{T}(\kls - k)$ as multiple Wiener-It\^{o} integrals.
\begin{proposition}\label{lemma_formal}
    Let $X_t = \int_0^t \me^{-k(t-s)}\dif G_s$ be a OU process driven by $G_t$. $\sqrt{T}(\kls-k)$ can be rewritten as:
    \begin{equation}
        \sqrt{T}(\kls-k)= \frac{I_0 +I_1(f_{1,T}) +I_2(f_{2,T})}{J_0 +I_1(h_{1,T})+J_2(h_{2,T})}.
    \end{equation}
    where
    \begin{subequations}
    \begin{align}
        I_0 &= \frac{1}{T^{3/2}}\Bigl(\mu^2 a_T\cdot d_T+q_T+\frac{\mu^2}{k}a_T^2+k e_T\Bigr)+\frac{\mu^2}{T^{1/2}}(\me^{-kT}-1),\label{formula_I_0}\\
        f_{1,T} &= \frac{\mu}{T^{3/2}}\bigl(d_T k_T-a_T l_T\bigr) +\frac{\mu}{\sqrt{T}}\bigl(m_T-k_T\bigr),\label{formula_I_1}\\
        f_{2,T} &= \frac{1}{T^{3/2}}l_T\otimes k_T+\frac{k}{T^{3/2}} l_T\otimes l_T-\frac{1}{2\sqrt{T}}f_T,\label{formula_I_2}\\
        J_0 &= \frac{c_T}{T} +b_T -\frac{1}{T^2}(\mu^2 d_T^2+e_T),\label{formula_J_0}\\
        h_{1,T} &= \frac{2\mu}{T}\bigl( l_T+ n_T\bigr) -\frac{2}{T^2}d_T l_T,\label{formula_J_1}\\
        h_{2,T} &=  g_T -\frac{1}{T^2}l_T\otimes l_T\label{formula_J_2}.
    \end{align}
    \end{subequations}
\end{proposition}
\begin{proof}
    We first deal with the numerator of \eqref{klsclt}:
    \begin{equation}
        \frac{V_T}{\sqrt{T}}\hat{\mu}+k\sqrt{T}\hat{\mu}(\hat{\mu}-\mu)-\frac{1}{\sqrt{T}}\int_0^T V_t\dif G_t.
        \label{numerator_1}
    \end{equation}    
    According to the definition of above functions, $V_T = \mu a_T +I_1(k_T)$.  Combining with \Cref{lemma_mutli_equa}, we can obtain
    \begin{align*}
        \frac{V_T}{\sqrt{T}}\hat{\mu} &= \frac{\mu a_T +I_1(k_T)}{T^{1/2}} \frac{ \mu d_T + I_1(l_T)}{T}\\
        &= \frac{1}{T^{3/2}} \bigl(\mu^2 a_T\cdot d_T + \mu a_T I_1(l_T) +\mu d_T I_1(k_T)+I_1(l_T)I_1(k_T)\bigr)\\
        &= \frac{1}{T^{3/2}} \Bigl(\mu^2 a_T\cdot d_T + \mu a_T I_1(l_T) +\mu d_T I_1(k_T)\\
        &\ \ \ +I_2(l_T\otimes k_T)+\normh{l_T\otimes_1 k_T}^2\Bigr).
    \end{align*}
    Then the second term of \eqref{numerator_1}. Let $Item_2:=k\sqrt{T}\hat{\mu}(\hat{\mu}-\mu)$, we have
    \begin{align*}
        Item_2 &= k\sqrt{T}(\hat{\mu}-\mu+\mu)(\hat{\mu}-\mu)\\
        &=k\sqrt{T}(\hat{\mu}-\mu)^2+k\sqrt{T}\mu(\hat{\mu}-\mu)\\
        &=\frac{k}{T^{3/2}}\Bigl(\frac{\mu^2}{k^2}(\me^{-kT}-1)^2+\frac{2\mu}{k}(\me^{-kT}-1)I_1(l_T)+I_2(l_T\otimes l_T)+e_T\Bigr)\\
        &\ \ \ +\frac{k\mu}{T^{1/2}}\Bigl(\frac{\mu}{k}(\me^{-kT}-1)+I_1(l_T)\Bigr)\\
        &=\frac{k}{T^{3/2}}\Bigl(\frac{\mu^2}{k^2}(\me^{-kT}-1)^2+\frac{2\mu}{k}(\me^{-kT}-1)I_1(l_T)+I_2(l_T\otimes l_T)+e_T\Bigr)\\
        &\ \ \ +\frac{\mu^2}{T^{1/2}}(\me^{-kT}-1)+\frac{\mu}{T^{1/2}}\bigl(G_T-I_1(k_T)\bigr).
    \end{align*}
    Besides, we can obtain
    \begin{align*}
        -\frac{1}{\sqrt{T}}\int_0^T V_t \dif G_t &= -\frac{1}{\sqrt{T}}\Bigl(\int_0^T \mu(1-\me^{-kt})\dif G_t+ \frac{1}{2} I_2(f_T)\Bigr)\\
        &= -\frac{1}{\sqrt{T}}\Bigl(\mu (G_T-I_1(m_T)) + \frac{1}{2} I_2(f_T)\Bigr).
    \end{align*}

    Next, we consider the denominator. For $\frac{1}{T}\int_0^T V_t^2\dif t$, we have
    \begin{align*}
        \frac{1}{T}\int_0^T V_t^2\dif t &= \frac{1}{T}\int_0^T (\mu(1-\me^{-kt})+X_t)^2 \dif t\\
        &= \frac{1}{T}\Bigl(\int_0^T \mu^2(1-\me^{-kt})^2\dif t + 2\int_0^t \mu(1-\me^{-kt})X_t\dif t +\int_0^T X_t^2\dif t\Bigr)\\
        &= \frac{1}{T}\Bigl(c_T +2\mu \int_0^T\int_0^t (1-\me^{-kt})\me^{-k(t-s)}\dif G_s\dif t\Bigr) +I_2(g_T)+b_T\\
        &= \frac{c_T}{T} + \frac{2\mu}{T}\bigl( I_1(l_T)+ I_1(n_T)\bigr) +I_2(g_T)+b_T.
    \end{align*}
    Since $\hat{\mu}^2=(\frac{1}{T}\int_0^T V_t\dif t)^2$, we can obtain
    \begin{align*}
        \Bigl(\frac{1}{T}\int_0^T V_t\dif t\Bigr)^2 &= \frac{1}{T^2}\Bigl(\mu d_T+I_1(l_T)\Bigr)^2\\
        &=\frac{1}{T^2}\Bigl(\mu^2 d_T^2+2\mu d_T I_1(l_T)+I_1^2(l_T)\Bigr)\\
        &=\frac{1}{T^2}\Bigl(\mu^2 d_T^2+2\mu d_T I_1(l_T)+I_2(l_T\otimes l_T)+e_T\Bigr).
    \end{align*}
    Combining the above formulas, we obtain the desired result.
\end{proof}

For simplicity, let $F_T:= (I_1(f_{1,T})+I_2(f_{2,T}))/(\sqrt{k}\sigma_\beta),\; H_T := J_0+I_1(h_{1,T})+I_2(h_{2,T})$. We show the convergence speed of zero-mean part. 
\begin{lemma}\label{lemma_kls}
    Let $Z\sim\mathcal{N}(0,1)$ be a standard Normal variable. Assume $\beta\in(1/2,3/4)$ and \Cref{hyp_1} holds. When $T$ is large enough, there exists constant $C_{\beta,V}^\prime$ such that
    \begin{equation}
        \sup_{z\in\Rnum}\abs{\Pro\left(\frac{F_T}{H_T}\leq z\right)-\Pro(Z\leq z)}\leq \frac{C_{\beta,V}^\prime}{T^\gamma},
    \end{equation}
    where $\gamma = \min{\{1/2, 3-4\beta\}}$. 
\end{lemma}
\begin{proof}
    According to \cite{peixingzhi} formulas (9) and (47),
    \begin{align*}
        \mE H_T= \mE J_0\to \alpha= C_\beta \Gamma(2\beta-1)k^{-2\beta}\;\;\; \mathrm{a.s.}.
    \end{align*}
    Combining with \cite{nourdin_normal_2012} Lemma 5.2.4, we can obtain 
    \begin{align}
        \max\{\normh{l_T}^2,\normh{n_T}^2\}\leq  C^{(1)} T^{2\beta},
        \label{norm_h}
    \end{align}
    where $C^{(1)}$ is a constant, and $\normh{h_{1,T}}^2\leq C^{(2)}/ T^{2-2\beta}$.

    Following from \cite{chen_parameter_2021} Theorem 1.4 and \eqref{norm_h}, we have
    \begin{align*}
        \norm{g_T}_{\mathfrak{H}^{\otimes 2}} \leq \frac{C^{(3)}}{\sqrt{T}},\;\;\; \norm{g_T\otimes_1 g_T}_{\mathfrak{H}^{\otimes 2}}\leq \frac{C^{(3)}}{T},\;\;\;\norm{\frac{2}{T^2}l_T\otimes l_T}_{\mathfrak{H}^{\otimes 2}}^2\leq \frac{C^{(3)}}{T^{4-4\beta}},
    \end{align*}
    where $C^{(3)}$ is a constant independent of $T$.
    Minkowski inequality ensures that  $\norm{h_{2,T}}_{\mathfrak{H}^{\otimes 2}}^2\leq C^{(4)}/ T$. Simalarly, the Proposition 3.10 of \cite{peixingzhi} induces that 
    \begin{align}
        J_0^2-\sum_{j=1}^m \frac{j!\norm{f_{j,T}}_{\mathfrak{H}^{\otimes j}}^2}{k\sigma_\beta^2}\leq C^{(5)}/T^{3-4\beta},
        \label{third_item}
    \end{align}
    and $\normh{f_{1,T}}^2\leq C^{(5)}/ T$. The formula (5.13) of\cite{chen_parameter_2021} ensures that 
    \begin{align*}
        \norm{f_{2,T}\tilde{\otimes}_{1} h_{2,T}}_{\mathfrak{H}^{\otimes 2}}\leq C^{(5)}/ T^{1/2}.
    \end{align*} 
    Combining the above results, \Cref{theorem_triple} and Cauchy-Schwarz inequality, we obtain the Lemma.
\end{proof}

The following Lemma shows the bound of non-zero mean part.
\begin{lemma}\label{lemma_kls1}
    Assume $\beta\in(1/2,3/4)$ and \Cref{hyp_1} holds. When $T$ is large enough, there exists constant $C_1$ such that
    \begin{equation*}
        \Pro\Bigl(\abs{L_1}> \frac{C_1}{T^{(3/4-\beta)}}\Bigr)=0,
    \end{equation*}
    where $L_1 = I_0/(\sqrt{k}\sigma_\beta H_T)$.
\end{lemma}
\begin{proof}
    The Proposition 3.14 and Corollary 3.15 of \cite{peixingzhi} ensure that when $T$ is large enough,
    \begin{equation}
        \frac{1}{T}\int_0^T V_t^2\dif t - \Bigl(\frac{1}{T}\int_0^T V_t\dif t\Bigr)^2\to C_\beta\Gamma(2\beta-1)k^{-2\beta}=\alpha\;\;\; \mathrm{a.s.}.
        \label{limit_J_0}
    \end{equation}
    Then we can obtain
    \begin{equation*}
        \abs{L_1}=\abs{\frac{\frac{1}{\sqrt{k}\sigma_\beta}I_0}{H_T}}\leq C^\prime\abs{I_0},
    \end{equation*}
    where  $C^\prime$ is a constant. Furthermore, According to \eqref{formula_I_0}, there exists $C^{\prime\prime}$ such that
    \begin{equation*}
        \abs{I_0} \leq \frac{1}{T^{3/2}}\Bigl(\mu^2 a_T\cdot d_T+q_T+\frac{\mu^2}{k}a_T^2+k e_T\Bigr)+\frac{\mu^2}{T^{1/2}}a_T\leq \frac{C^{\prime\prime}}{T^{3/2-2\beta}}.
    \end{equation*}
    Combining the above two formulas, we have
    \begin{equation*}
        \abs{L_1}\leq \frac{C^{(6)}}{T^{3/2-2\beta}}\;\;\; \mathrm{a.s.},
    \end{equation*}
    where $C^{(6)}=2C^{\prime}\cdot C^{\prime\prime}$. Then we obtain the desired result.
\end{proof}
We now prove the formula \eqref{hat_kls}.
\begin{proof}[Proof of formula \eqref{hat_kls}]
    According to \Cref{chang_lemma},
    \begin{align*}
        \sup_{z\in\Rnum}\abs{\Pro\Bigl(\sqrt{\frac{T}{k\sigma_\beta^2}}(\kls-k)\leq z\Bigr)-\Pro(Z\leq z)}&\leq \sup_{z\in\Rnum}\abs{\Pro\Bigl(\frac{F_T}{H_T}\leq z\Bigr)-\Pro(Z\leq z)}\\
        &\ \ \ +\Pro\Bigl(\abs{L_1}>\frac{C_1}{T^{3/4-\beta}}\Bigr)+\frac{1}{\sqrt{2\pi}}\frac{C_1}{T^{3/4-\beta}}.
    \end{align*}
    Combining Lemmas \ref{lemma_kls} and \ref{lemma_kls1}, we obtain the desired result.
\end{proof}


Pei \etal \cite{peixingzhi} show the CLT of least squares estimator $\muls$ of mean coefficient $\mu$. 

\begin{theorem}[\cite{peixingzhi} Propositions 4.21]
    Assume $\beta\in(1/2,1)$ and $G_t$ is a self-similar Gaussian process satisfing \Cref{hyp_1} and $\mE[G_1^2] =1$. $T^{1-\beta}(\muls - \mu)$ is asymptotically normal as $T\to\infty$: 
    \begin{equation}
        T^{1-\beta}(\muls - \mu) =\frac{\frac{V_{T}}{T^{\beta}}\big[\frac{1}{T}\int_{0}^{T}V_{t}^{2}\dif t- \mu \hat{\mu}\big] -\frac{1}{T}\int_{0}^{T}V_{t}\dif V_{t}\cdot T^{1-\beta}\big[ \hat{\mu}-\mu \big]}{\frac{V_{T}}{T}\cdot  \hat{\mu}- \frac{1}{T} \int_{0}^{T}V_{t}\dif V_{t}} \distribution \mathcal{N}(0,1/k^2).
        \label{mulsclt}
    \end{equation}
\end{theorem}

We also transform $T^{1-\beta}(\muls - \mu)$ as the following multiple integrals.
\begin{proposition}\label{prop_muls}
    $T^{1-\beta}(\muls-\mu)$ can be represented by as :
    \begin{equation*}
        T^{1-\beta}(\muls-\mu) = \frac{I_0^* +I_1(f_{1,T}^*) +I_2(f_{2,T}^*)+I_3(f_{3,T}^*)}{J_0^* +I_1(h_{1,T}^*)+I_2(h_{1,T}^*)},
    \end{equation*}
    where
    \begin{small}
    \begin{subequations}
        \begin{align}
            &I_0^* = \frac{1}{T^{1+\beta}}\Bigl(2\mu\bigl(\normh{k_T\otimes_1 l_T}^2+\normh{k_T\otimes_1 n_T}^2\bigr) +2k\mu\normh{n_T\otimes_1 l_T}^2 +\mu\normh{m_T\otimes_1 l_T}^2\Bigr), \label{eq_i_0^*}\\
            &f_{1,T}^* = \frac{1}{T^{1+\beta}}\Bigl((c_T -\mu^2 d_T-\frac{\mu^2 a_T}{k}\mathbbm{1}_{[0,T]}) +2\mu^2 a_T l_T+f_T\otimes_1 l_T\Bigr)\notag\\
            &\; \; \;\;\;\; -\frac{1}{T^{1+\beta}}\frac{\mu^2 a_T}{k} m_T+\frac{1}{T^{\beta}}\Bigl(b_T \mathbbm{1}_{[0,T]} +2g_T\otimes_1 \mathbbm{1}_{[0,T]}\Bigr),\label{eq_i_1^*2}\\
            &f_{2,T}^* = \frac{1}{T^{1+\beta}}\Bigl(2\mu \bigl(k_T\otimes l_T+k_T\otimes n_T \bigr)+2k\mu (n_T\otimes l_T)+\mu m_T\otimes l_T -\frac{\mu}{2k}a_T f_T\Bigr),\label{eq_i_2^*}\\
            &f_{3,T}^* = \frac{1}{T^\beta}\bigl(g_T \otimes k_T+kg_T \otimes l_T\bigr)+\frac{1}{2T^{1+\beta}}f_T \otimes l_T,\label{eq_i_3^*}\\
            &J_0^* = \frac{1}{T^2}\Bigl(\mu^2 a_T\cdot d_T+\normh{l_T\otimes_1 k_T}^2\Bigr)+\frac{1}{T}\Bigl(kc_T-k\mu^2d_T\bigr)+kb_T,\label{eq_j_0^*}\\
            &h_{1,T}^* = \frac{1}{T^{2}}\Bigl(\mu a_T l_T +\mu d_T k_T\Bigr)+ \frac{1}{T}\Bigl(\mu k_T+ 2k\mu n_T+\mu m_T\Bigr),\label{eq_j_1^*}\\
            &h_{2,T}^* =\frac{1}{T^{2}}l_T\otimes k_T-\frac{1}{2T} f_T+kg_T\label{eq_j_2^*}.
        \end{align}
    \end{subequations}
    \end{small}
\end{proposition}
\begin{proof}
    We first consider the denominator.
    \begin{align*}
        \frac{V_{T}}{T}\cdot  \hat{\mu}- \frac{1}{T} \int_{0}^{T}V_{t}\dif V_{t}.
    \end{align*}
    According to \Cref{lemma_formal},
    \begin{align*}
        \frac{V_{T}}{T}\hat{\mu} &= \frac{\mu a_T +I_1(k_T)}{T} \frac{ \mu d_T + I_1(l_T)}{T}\\
        &= \frac{1}{T^{2}} \bigl(\mu^2 a_T\cdot d_T + \mu a_T I_1(l_T) +\mu d_T I_1(k_T)+I_1(l_T)I_1(k_T)\bigr)\\
        &= \frac{1}{T^{2}} \Bigl(\mu^2 a_T\cdot d_T + \mu a_T I_1(l_T) +\mu d_T I_1(k_T)+I_2(l_T\otimes k_T)+\normh{l_T\otimes_1 k_T}^2\Bigr). 
    \end{align*}
    Then, we deal with $-\frac{1}{T}\int_{0}^{T}V_{t}\dif V_{t}$:
    \begin{align*}
        -\frac{1}{T}\int_{0}^{T}V_{t}\dif V_{t} &= -\frac{1}{T}\Bigl(\int_{0}^{T} k\mu V_t \dif t -\int_0^T kV_t^2\dif t +\int_0^T  V_t\dif G_t\Bigr)\\
        &= -\frac{k\mu}{T}\int_{0}^{T} V_t\dif t +\frac{k}{T}\int_0^T V_t^2\dif t -\frac{1}{T}\int_0^T  V_t\dif G_t\\
        &= -\frac{k\mu^2  d_T + k\mu I_1(l_T)}{T}+\Bigl(\frac{kc_T}{T} + \frac{2k\mu}{T}\bigl( I_1(l_T)+ I_1(n_T)\bigr)\Bigr)\\
        &\ \ \ -\frac{1}{T}\Bigl(\mu \bigl(G_T-I_1(m_T)\bigr) + \frac{1}{2} I_2(f_T)\Bigr) +kI_2(g_T)+kb_T\\
        &= \frac{1}{T}\Bigl(kc_T - k\mu^2  d_T +2k\mu I_1(n_T)+\mu I_1(k_T) +\mu I_1(m_T)+\frac{1}{2} I_2(f_T)\Bigr)\\
        &\ \ \ +kI_2(g_T)+kb_T.
    \end{align*}
    
    We next consider the numerator:
    \begin{align*}
        \frac{V_{T}}{T^{\beta}}\Big[\frac{1}{T}\int_{0}^{T}V_{t}^{2}\dif t- \mu \hat{\mu}\Big] -\frac{1}{T}\int_{0}^{T}V_{t}\dif V_{t}\cdot T^{1-\beta}\big[ \hat{\mu}-\mu \big].
    \end{align*}
    Since the \Cref{lemma_formal}, we have
    \begin{small}
    \begin{align*}
        \frac{V_{T}}{T^{\beta}}\frac{1}{T}\int_{0}^{T}V_{t}^{2}dt &= \frac{1}{T^{\beta}}\bigl(\mu a_T +I_1(k_T)\bigr)\Bigl(\frac{c_T}{T} + \frac{2\mu}{T}\bigl(I_1(l_T)+ I_1(n_T)\bigr) +I_2(g_T)+b_T\Bigr)\\
        &=\frac{1}{T^{1+\beta}}\Bigl(\mu a_T \cdot c_T +2\mu^2 a_T \bigl(I_1(l_T)+I_1(n_T)\bigr)+c_T I_1(k_T)\\
        &\; \; \; \; \; \; \; \; \; \; \; \; \; \; \; +2\mu \bigl(I_2(k_T\otimes l_T)+I_2(k_T\otimes n_T) \bigr)\\
        &\; \; \; \; \; \; \; \; \; \; \; \; \; \; \; +2\mu\bigl(\normh{k_T\otimes_1 l_T}^2+\normh{k_T\otimes_1 n_T}^2\bigr)\Bigr)\\
        &\; \; \; +\frac{1}{T^{\beta}}\Bigl(\mu a_T I_2(g_T)+\mu a_T b_T+I_3(k_T\otimes g_T)\\
        &\; \; \; \; \; \; \; \; \; \; \; \; \; \; \; + 2I_1(k_T\otimes_1 g_T)+b_T I_1(k_T)\Bigr).
    \end{align*}
    \end{small}
    Similarly, we can obtain
    \begin{align*}
        -\frac{V_{T}}{T^{\beta}}\cdot\mu \hat{\mu} &= -\mu\frac{\mu a_T +I_1(k_T)}{T^{\beta}} \frac{ \mu d_T + I_1(l_T)}{T}\\
        &= \frac{-\mu}{T^{1+\beta}} (\mu^2 a_T\cdot d_T + \mu a_T I_1(l_T) +\mu d_T I_1(k_T)+I_1(l_T)I_1(k_T))\\
        &= \frac{-\mu}{T^{1+\beta}} \Bigl(\mu^2 a_T\cdot d_T + \mu a_T I_1(l_T) +\mu d_T I_1(k_T)\\
        &\ \ \ +I_2(l_T\otimes k_T)+\normh{l_T\otimes_1 k_T}^2\Bigr). 
    \end{align*}
    Let $Item_3 = -\frac{1}{T}\int_0^T V_t\dif V_t \cdot T^{1-\beta} (\hat{\mu}-\mu)$, we have
    \begin{align*}
        Item_3&=\Bigl(\frac{I_1(l_T)}{T^\beta}-\frac{\mu a_T}{kT^\beta}\Bigr)\frac{1}{T}\int_{0}^{T}-V_{t}\dif V_{t} \\
        &= \frac{1}{T^{1+\beta}}\biggl(kc_T I_1(l_T)-k\mu^2 d_T I_1(l_T) +2k\mu I_2(n_T\otimes l_T)+2k\mu\normh{n_T\otimes_1 l_T}^2\\
        &\ \ \ +\mu I_2(k_T\otimes l_T)+\mu\normh{k_T\otimes_1 l_T}^2 +\mu I_2(m_T\otimes l_T)+\mu\normh{m_T\otimes_1 l_T}^2 \biggr)\\
        &+ \frac{k}{T^{\beta}} \Bigl(b_T I_1(l_T)+I_3(g_T\otimes l_T)+2 I_1(g_T\otimes_1 l_T)\Bigr) \\
        &+ \frac{1}{2T^{1+\beta}} \Bigl(I_3(f_T\otimes l_T)+2 I_1(f_T\otimes_1 l_T)\Bigr)\\
        &+\frac{1}{T^{1+\beta}}\Bigl(\mu^3 a_T\cdot d_T - \mu a_T\cdot c_T -2\mu^2 a_T I_1(n_T)\Bigr)\\
        &- \frac{1}{kT^{1+\beta}}\Bigl(\mu^2 a_T I_1(k_T)+\mu^2 a_T I_1(m_T) +\frac{1}{2}\mu a_T I_2(f_T) \Bigr)\\
        &-\frac{1}{T^{\beta}}\bigl(\mu a_T \cdot b_T +\mu a_T I_2(g_T)\bigr).
    \end{align*}
    Combining the above formulas, we obtain the Proposition.
\end{proof}

Denote $F_T^*:=I_0^* +I_1(f_{1,T}^*) +I_2(f_{2,T}^*)+I_3(f_{3,T}^*)$ and $H_T^*= \frac{1}{k}(J_0^* +I_1(h_{1,T}^*)+I_2(h_{1,T}^*))$.
We now prove the upper bounds of zero-mean part.
\begin{lemma}\label{lemma_muls}
    Let $Z\sim\mathcal{N}(0,1)$ be a standard Normal variable. Assume $\beta\in(1/2,1)$ and $G_t$ is a self-similar Gaussian process satisfing \Cref{hyp_1} and $\mE[G_1^2] =1$. When $T$ is large enough, there exists a constant $C_{\beta,V}^\prime$ such that
    \begin{equation}
        \sup_{z\in\Rnum}\abs{\Pro\left(\frac{F_T^*}{H_T^*}\leq z\right)-\Pro(Z\leq z)}\leq \frac{C_{\beta,V}^\prime}{T^{2-2\beta}},
    \end{equation}
\end{lemma}
\begin{proof}
    According to \Cref{lemma_kls}, \eqref{eq_j_1^*} and \eqref{eq_j_2^*}, we have
    \begin{align*}
        \normh{h_{1,T^*}}^2\leq \frac{K_1}{T^{2-2\beta}},\;\;\;\norm{h_{2,T^*}}_{\mathfrak{H}^{\otimes 2}}^2\leq \frac{K_1}{T},
    \end{align*}
    where $K_1$ is a constant independent of $T$. It is easy to see that \Cref{lemma_kls} and equation \eqref{eq_j_0^*} implies $\mE H_T^* \to\alpha\;\mathrm{a.s.}$. Furthermore, when $T$ is large enough, there exists constant $K_2$ such that
    \begin{align}
        \abs{(\mE H_T^*)^2-\alpha^2}&\leq \frac{K_2}{T^{2-2\beta}}.
    \end{align}
    Since $G_t$ is self-similar, \Cref{lemma_kls} implies $\abs{\normh{f_{1,T}^*}^2 -\alpha^2}\leq \frac{K_3}{T^{2-2\beta}}$. Also, we can obtain
    \begin{align}
        \norm{f_{2,T}^*}_{\mathfrak{H}^{\otimes 2}}^2\leq \frac{K_3}{T^2},\;\;\;\norm{f_{3,T}^*}_{\mathfrak{H}^{\otimes 3}}^2\leq \frac{K_3}{T}
    \end{align}
    Combining above three formulas, we have 
    \begin{align*}
        J_0^{*,2}-\sum_{j=1}^m j!\norm{f_{j,T}^*}_{\mathfrak{H}^{\otimes j}}^2\leq K_4/T^{2-2\beta},
    \end{align*}
    Following from \Cref{theorem_triple} and Cauchy-Schwarz inequality, we obtain the result.
\end{proof}

We next consider the non-zero mean part.

\begin{lemma}\label{lemma_muls1}
    Assume $\beta\in(1/2,1)$ and $G_t$ is a self-similar Gaussian process satisfing \Cref{hyp_1} and $\mE[G_1^2] =1$. Let $M_1= I_0^*/H_T^*$. When $T$ is large enough, there exists a constant $C_1$ independent of $T$ such that
    \begin{equation*}
        \Pro\Bigl(\abs{M_1}> \frac{C_1}{T^{1/2-\beta/2}}\Bigr)=0.
    \end{equation*}
\end{lemma}
\begin{proof}
    Following from \Cref{lemma_muls}, we have $\abs{M_1}\leq C^\prime\abs{I_0^*}$, where $C^\prime$ is a constant independent of $T$. Furthermore, Equation \eqref{eq_i_0^*} implies that there exists $C^{\prime\prime}$ such that
    \begin{align*}
        \abs{I_0^*} \leq \frac{C^{\prime\prime}}{T^{1-\beta}},\;\;\;\abs{M_1}\leq \frac{C_1}{T^{1-\beta}}\;\;\; \mathrm{a.s.},
    \end{align*}
    where $C_1=2C^{\prime}\cdot C^{\prime\prime}$. Combining the above formulas, we obtain the desired result.
\end{proof}

We now prove the formula \eqref{thm_muls}.
\begin{proof}[Proof of formula \eqref{thm_muls}]
    Following from \Cref{chang_lemma}, we have
    \begin{align*}
        \sup_{z\in\Rnum}\abs{\Pro\Bigl(\frac{k}{T^{\beta-1}}(\muls-\mu)\leq z\Bigr)-\Pro(Z\leq z)}&\leq \sup_{z\in\Rnum}\abs{\Pro\Bigl(\frac{F_T^*}{H_T^*}\leq z\Bigr)-\Pro(Z\leq z)}\\
        &\ \ \ +\Pro\Bigl(\abs{M_1}>\frac{C_1}{T^{1/2-\beta/2}}\Bigr)\\
        &\ \ \ +\frac{1}{\sqrt{2\pi}}\frac{C_1}{T^{1/2-\beta/2}}.
    \end{align*}
    Combining Lemmas \ref{lemma_muls} and \ref{lemma_muls1}, we obtain the formula \eqref{thm_muls}.
\end{proof}

\section*{Acknowledgments}
We gratefully acknowledge the very valuable suggestions by referees. Y. Chen is supported by National Natural Science Foundation of China (NSFC) with grant No.11961033.

\section*{Declarations}
\begin{itemize}
    \item Availability of data and materials
\end{itemize}

This manuscript has no associated data.
\nocite{*}
\bibliography{Vasicek_model}


\end{document}